\documentclass[12pt]{amsart}
\usepackage{fullpage}
\usepackage{amsmath, amsthm, amssymb}
\usepackage{enumerate}
\usepackage[centertableaux]{ytableau}
\usepackage{todonotes}

\usepackage{tikz}
\tikzstyle{v}=[circle, draw, fill=white,
inner sep=0pt, minimum width=4pt]
\tikzstyle{w}=[circle, draw, fill=black,
inner sep=0pt, minimum width=4pt]
\newcommand{\dr}[1]{\rule[-3ex]{0pt}{0pt}\rule[4ex]{0pt}{0pt} \raisebox{\dimexpr-.5\height+.5\ht\strutbox\relax}{\tikz\draw#1;}}

\newtheorem{thm}{Theorem}[section]

\newtheorem{prop}[thm]{Proposition}
\newtheorem{cor}[thm]{Corollary}

\theoremstyle{definition}

\newtheorem*{defn}{Definition}
\newtheorem{ex}[thm]{Example}

\theoremstyle{remark}
\newtheorem{rmk}[thm]{Remark}

\newcommand{\NN}{\mathbf N}
\newcommand{\CC}{\mathbf C}
\newcommand{\Sym}{\mathfrak S}
\newcommand{\sgn}{\operatorname{sgn}}
\newcommand{\Res}{\operatorname{Res}}
\newcommand{\ov}{\overline}
\newcommand{\Stab}{\operatorname{Stab}}
\newcommand{\id}{\mathrm{id}}
\newcommand{\by}{\begin{ytableau}}
\newcommand{\ey}{\end{ytableau}}
\newcommand{\yts}{\ytableaushort}
\renewcommand{\subset}{\subseteq}

\begin{document}

\title{Complete branching rules for Specht modules}
\author{Ricky Ini Liu}
\email{riliu@ncsu.edu}
\address{Department of Mathematics, North Carolina State University, Raleigh, NC 27695}
\maketitle

\begin{abstract}
We give a combinatorial description for when the Specht module of an arbitrary diagram admits a (complete) branching rule. This description, given in terms of the maximal rectangles of the diagram, generalizes all previously known branching rules for Specht modules, such as those given by Reiner and Shimozono for northwest diagrams and by the present author for forest diagrams.
\end{abstract}

\section{Introduction}

One of the basic facts about the representation theory of the symmetric group $\Sym_n$ is that the irreducible representations $S^\lambda$ (where $\lambda$ is a partition of $n$) satisfy a branching rule: there is a multiplicity-free description of the restriction of an irreducible representation from $\Sym_n$ to $\Sym_{n-1}$ via
\[\Res^{\Sym_n}_{\Sym_{n-1}} S^\lambda \cong \bigoplus_\mu S^\mu,\]
where the direct sum ranges over all partitions $\mu$ that can be obtained from $\lambda$ by removing a corner box. This fact forms the basis for much of the combinatorial structure behind the representation theory of $\Sym_n$. (For references, see, for instance, \cite{Fulton, OkounkovVershik, Sagan}.)

One of the standard constructions for $S^\lambda$ is as a \emph{Specht module}, an ideal within the group algebra $\CC[\Sym_n]$ generated by the \emph{Young symmetrizer} of the Young diagram of $\lambda$. This construction can be generalized to diagrams other than partitions by defining the Young symmetrizer in an analogous fashion. In the case of a skew Young diagram $\lambda/\mu$, the resulting skew Specht module $S^{\lambda/\mu}$ is well understood \cite{JamesPeel}. Similarly, the Specht module for the Rothe diagram of a permutation occurs in the theory of reduced decompositions and Stanley symmetric functions \cite{Kraskiewicz}. Specht modules $S^D$ for other classes of diagrams $D$ have also been studied (see, for instance, \cite{Liu, Magyar, ReinerShimozono1, ReinerShimozono2}).

For general diagrams $D$, the structure of $S^D$ is not well understood. However, in some special cases, it is known that $S^D$ satisfies a similar branching rule as the one for partitions, namely
\[\Res^{\Sym_n}_{\Sym_{n-1}} S^D \cong \bigoplus_{x \in B} S^{D \backslash \{x\}}, \tag{$*$} \label{eq-branch}\]
where $B \subset D$ is a specially chosen \emph{branching set}. In particular, such a branching rule is known to exist if $D$ belongs to the class of \emph{northwest diagrams} \cite{ReinerShimozono2} or \emph{forest diagrams} \cite{Liu}. As these two classes are rather distinct and the proofs in both cases are quite different, a natural question is whether these results can be unified with a single explanation. 

The main result of this paper gives a combinatorial characterization of when a diagram $D$ \emph{branches completely}, that is, when there exists a subset $B \subset D$ such that \eqref{eq-branch} holds (along with a technical condition) and $D \backslash \{x\}$ also branches completely for each $x \in B$. The essential condition is that $B$ must contain exactly one box from each maximal rectangle of $D$ (see Theorem~\ref{thm-main}). From this characterization, it is easy to deduce all previously known branching rules for general Specht modules---in particular, we will show how to deduce those from \cite{ReinerShimozono2} and \cite{Liu}. We will also show that the complete set of relations defining $S^D$ for any diagram $D$ that branches completely is generated by certain \emph{generalized Garnir relations}, which are multi-column analogues of the two-column Garnir relations that generate the relations for irreducible $S^\lambda$. 

We begin in \S2 with some background about Specht modules and known branching rules. In \S3 we describe the generalized Garnir relations and show the existence of certain maps between Specht modules of different shapes. In \S4, we will use these results to prove our main result on completely branching diagrams. We will also show how to deduce the branching rules of \cite{ReinerShimozono2} and \cite{Liu} from our main result. Finally, in \S5 we end with some concluding remarks and questions.

\section{Preliminaries}
We begin by defining some notation and reviewing some background about symmetric group representations. See, for instance, \cite{Sagan} for more information.

\subsection{Specht modules}
A \emph{diagram} $D$ is a finite subset of $\NN \times \NN$. We associate $\NN \times \NN$ with the boxes in a grid, where $(i,j)$ denotes the box in the $i$th row from the top and the $j$th column from the left. At times, it may be convenient to think of $D$ as a graph in the following way: given a diagram $D$, construct an associated bipartite graph $G=G(D)$ with vertex set $\{v_1, v_2, \dots\} \cup \{w_1, w_2, \dots\}$ in which $v_i$ is adjacent to $w_j$ if $(i,j) \in D$.

Given a diagram $D$, a \emph{tableau} $T$ of shape $D$ is a labeling of the boxes of $D$ with nonnegative integers. We say $T$ is \emph{injective} if each label $1, 2, \dots, n$ appears exactly once, where $n=|D|$. We say a tableau is \emph{row-strict} (resp. \emph{column-strict}) if each label appears at most once in a row (resp. column). If $x = (i,j) \in D$, we will write $T_x = T_{ij}$ for the label of $x$ in $T$.

A \emph{partition} $\lambda = (\lambda_1, \lambda_2, \dots, \lambda_k)\vdash n$ is a sequence of nonincreasing positive integers summing to $n$. The \emph{Young diagram} of a partition, which we also denote by $\lambda$, is the set of boxes $(i, j)$ with $1 \leq j \leq \lambda_k$ for $1 \leq i \leq k$. 

Let $\Sym_n$ be the symmetric group on $n$ letters and $\CC[\Sym_n]$ its group algebra. To any diagram $D$ with $n$ boxes, we construct a $\CC[\Sym_n]$-module as follows. Label the boxes of $D$ arbitrarily with the numbers $1, \dots, n$, and let $\Sym_n$ act on the boxes in the natural way. Let $R_D \subset \Sym_n$ be the subgroup that stabilizes each row of $D$, and let $C_D \subset \Sym_n$ be the subgroup that stabilizes each column of $D$. Then define $R(D), \ov C(D) \subset \CC[\Sym_n]$ by
\begin{align*}
R(D) &= \sum_{\sigma \in R_D} \sigma, & \ov C(D) &= \sum_{\sigma \in C_D} \sgn(\sigma) \cdot \sigma.
\end{align*}
Then the \emph{Specht module} $S^D$ is the left ideal $\CC[\Sym_n]\ov C(D)R(D)$. The Specht modules $S^\lambda$ for $\lambda \vdash n$ are precisely the irreducible representations of $S_n$. Note that permuting rows or permuting columns of $D$ does not change $S^D$ up to isomorphism. If two diagrams can be obtained from one another in this way, we will say that they are \emph{equivalent}.

Typically when working with a diagram $D$, we will think of elements of $\CC[\Sym_n]$ as formal linear combinations of injective tableaux of shape $D$. Then multiplication on the right by an element of $\Sym_n$ corresponds to applying a permutation to the boxes of $D$, while multiplication on the left corresponds to applying a permutation to the labels $1, \dots, n$. (Note that the right action of $\Sym_n$ can be applied to any tableau, not just injective tableaux.) Given an injective tableau $T$ of shape $D$, we will write $e_T = T \cdot \ov C(D)R(D) \in S^D$. 

\subsection{Branching}

The irreducible representations $S^\lambda$ satisfy the following \emph{branching rule}. A \emph{corner box} of $\lambda$ is a box $x \in \lambda$ such that $\lambda \backslash \{x\}$ is also the Young diagram of a partition. Equivalently, $x$ is the rightmost box of $\lambda$ in its row and the bottommost box of $\lambda$ in its column.

If $B$ is the set of corner boxes of $\lambda$, then the restriction of $S^\lambda$ from $\Sym_n$ to $\Sym_{n-1}$ (where $\Sym_{n-1}$ acts on the first $n-1$ letters in the usual way) can be decomposed into irreducibles as
\[\Res^{\Sym_n}_{\Sym_{n-1}} S^\lambda \cong \bigoplus_{x \in B} S^{\lambda \backslash \{x\}}.\]

Although it is not known in general how to decompose $S^D$ into irreducible representations, the most general known result for \emph{percentage-avoiding diagrams} is due to Reiner and Shimozono \cite{ReinerShimozono2}. This description is then used to give a branching rule for a subclass of diagrams called \emph{northwest diagrams}.

\begin{defn}
A diagram $D$ is \emph{northwest} if whenever $(i_1, j_2), (i_2, j_1) \in D$ for some $i_1<i_2$ and $j_1<j_2$, we also have $(i_1,j_1) \in D$.

A diagram $D$ has its rows in \emph{initial segment order} if whenever row $i_1$ (thought of as a subset of $\NN$) is a proper initial segment of row $i_2$, we have $i_1 < i_2$.  
\end{defn}
It is shown in \cite{ReinerShimozono2} that one can always rearrange the rows of a northwest diagram $D$ to form a new northwest diagram with rows in initial segment order. The following theorem summarizes the branching rule for northwest diagrams.
\begin{thm}[\cite{ReinerShimozono2}]\label{thm-northwest}
Let $D$ be a northwest diagram with $n$ boxes whose rows are in initial segment order. Let $B$ be the set of boxes $x$ such that $x$ is bottommost in its column, and no box $y$ that is bottommost in its column lies to the left of $x$ (in the same row). Then
\[\Res^{\Sym_n}_{\Sym_{n-1}} S^D \cong \bigoplus_{x \in B} S^{D \backslash \{x\}}.\]
\end{thm}

Theorem~\ref{thm-northwest} generalizes known branching rules for straight and skew Young diagrams, as well as column-convex diagrams \cite{ReinerShimozono1} and permutation diagrams \cite{Kraskiewicz}.

Separately, in \cite{Liu}, a branching rule is given for \emph{forest diagrams}.
\begin{defn}
A diagram $D$ is a \emph{forest diagram} if the corresponding graph $G(D)$ is a forest.

An \emph{almost perfect matching} $M$ of a forest $G$ is a set of edges that contains every isolated edge as well as exactly one edge incident to any vertex of degree greater than 1.
\end{defn}
It is shown in \cite{Liu} that a forest always has an almost perfect matching and that the following branching rule holds.
\begin{thm}[\cite{Liu}]\label{thm-forest}
Let $D$ be a forest diagram with $n$ boxes, and let $B$ be the set of boxes corresponding to any almost perfect matching of $G(D)$. Then
\[\Res^{\Sym_n}_{\Sym_{n-1}} S^D \cong \bigoplus_{x \in B} S^{D \backslash \{x\}}.\]
\end{thm}

\begin{ex}
The diagram shown below is both a northwest diagram with rows in initial segment order and a forest diagram. If $B$ is the set of shaded boxes, then both Theorem~\ref{thm-northwest} and Theorem~\ref{thm-forest} apply.
\[\ydiagram{1,3,1+1,2+2} * [*(lightgray)]{0,1,1+1,2+1}\]
\end{ex}

In general, the classes of northwest and forest diagrams do not contain one another. Therefore, we would like to find a common generalization of these two classes. In fact, both Theorem~\ref{thm-northwest} and Theorem~\ref{thm-forest} will follow as a consequence of the more general Theorem~\ref{thm-main} proved below, as we will see in \S4.

\subsection{Garnir relations}
By definition, $S^\lambda$ is linearly spanned by $e_T$ as $T$ ranges over injective tableaux of shape $\lambda$. The simplest relations among the $e_T$ are the one-column relations: for any $\sigma \in C_D$, \[e_{T\sigma} = T\cdot \sigma\ov C(D) \cdot R(D) = T \cdot \sgn(\sigma)\ov C(D) \cdot R(D) = \sgn(\sigma) \cdot e_T.\]

The only additional relations between the $e_T$ in $S^\lambda$ are generated by the \emph{Garnir relations}. These are two-column relations, that is, they describe a relation between some $e_T$ that differ only in two columns.

One description of the Garnir relations can be given as follows. Pick any two columns of $\lambda$, and suppose that they contain boxes in fewer than $a$ distinct rows. (Since $\lambda$ is a partition, $a$ will be larger than the length of the longer column.) Let $A$ be any set of $a$ boxes in these two columns, and let $\Sym_A \subset \Sym_n$ be the subgroup of permutations of $A$. Then for all $T$,
\[\sum_{\pi \in \Sym_A} \sgn(\pi) \cdot e_{T\pi} = 0.\]

\begin{rmk}\label{rmk-mcr}
Using the one-column relations, we can instead sum $\pi$ only over a set of left coset representatives for $\Sym_A \cap C_D$ in $\Sym_A$. This will only change the left hand side by a constant factor and will not affect the arguments below since we work in characteristic 0.
\end{rmk}

\begin{ex}
Let $A$ be the set of shaded boxes in the following tableau:
\[T_1 = \yts{14,25,36} * [*(lightgray)]{1+1,2,1}\,.\]
The corresponding Garnir relation (summing over a set of coset representatives for $\Sym_A \cap C_D$ in $\Sym_A$ as described in Remark \ref{rmk-mcr}) is $e_{T_1} - e_{T_2} + e_{T_3} + e_{T_4} - e_{T_5} + e_{T_6}=0$, where
\[T_2 = \yts{13,25,46} * [*(lightgray)]{1+1,2,1}\,, \qquad
T_3 = \yts{13,24,56} * [*(lightgray)]{1+1,2,1}\,, \qquad
T_4 = \yts{12,35,46} * [*(lightgray)]{1+1,2,1}\,, \qquad
T_5 = \yts{12,34,56} * [*(lightgray)]{1+1,2,1}\,, \qquad
T_6 = \yts{12,43,56} * [*(lightgray)]{1+1,2,1}\,.\]  
\end{ex}

Although the Garnir relations are usually stated for partitions, they also hold for more general diagrams, though they usually do not generate all relations in $S^D$. We will prove that a generalized multi-column relation holds for general Specht modules below in Proposition~\ref{prop-relations}.

\section{Specht modules}

In this section, we will formulate a general notion of branching for diagrams. We will also construct multi-column analogues of the Garnir relations and show the existence of certain maps between Specht modules of different shapes.

\subsection{Special transversals}
We first define a property of a set of boxes that will allow it to play the role of the corner boxes for a general diagram. 
\begin{defn}
Given a diagram $D$, a subset $B \subset D$ is a \emph{special transversal} if the elements of $B$ can be ordered $b_1 = (i_1, j_1), b_2 = (i_2, j_2), \dots$ such that $(i_p, j_q) \not \in D$ for any $p<q$.
\end{defn}
In particular, no two boxes of a special transversal lie in the same row or column. Note that the corner boxes of a partition form a special transversal when ordered from bottom to top. The branching boxes in the rules given for northwest and forest diagrams in \cite{ReinerShimozono2} and \cite{Liu} also form special transversals.

Special transversals also have a graph-theoretic definition, as shown in \cite{Liu}.
\begin{prop} [\cite{Liu}] \label{prop-alternating}
	Let $B \subset D$, and let $M$ be the corresponding set of edges in $G=G(D)$. Then $B \subset D$ is a special transversal if and only if $M$ is a matching of $G$ (that is, no two edges of $M$ share a vertex), and $G$ contains no $M$-alternating cycle (that is, no cycle in $G$ has half of its edges in $M$).
\end{prop}

The reason that the condition for being a special transversal is a natural one is the following proposition. (Similar results appear in \cite{Kraskiewicz, Liu}, but we include the proof as it will be useful later.)

\begin{prop} \label{prop-filtration}
Let $B = \{b_1, b_2, \dots, b_k\}$ be a special transversal of a diagram $D$, and let $n=|D|$. For $i = 1, \dots, k$, let $V_i$ be the linear subspace of $S^D$ spanned by $e_T$, where $T$ contains the label $n$ in box $b_j$ for some $j \geq i$. Then there exists a filtration
\[0 = V_{k+1} \subset V_k \subset V_{k-1} \subset \dots \subset V_2 \subset V_1 \subset S^D\]
and surjective $\Sym_{n-1}$-homomorphisms $\varphi_i\colon V_i \to S^{D \backslash \{b_i\}}$ such that $V_{i+1} \subset \ker \varphi_i$ for $i=1, \dots, k$.
\end{prop}
\begin{proof}
We assume without loss of generality that $b_i = (i,i)$.

For any injective tableau $T$ of shape $D$, define $\varphi_i(T)$ to be the tableau obtained by removing box $b_i=(i,i)$ from $T$ if $T_{ii} = n$, otherwise let $\varphi_i(T)=0$. Then $\varphi_i$ defines a $\Sym_{n-1}$-module homomorphism from $\CC[\Sym_n]$, spanned by tableaux of shape $D$, to $\CC[\Sym_{n-1}]$, spanned by tableaux of shape $D \backslash \{b_i\}$.

Suppose $T$ is an injective tableau with $T_{jj} = n$ for some $j>i$. Any tableau of the form $T\sigma$ for $\sigma \in C_D$ contains $n$ in column $j$, so by the definition of a special transversal, it cannot have $n$ in row $i$. Hence $e_T = T \ov C(D) R(D)$ cannot have any term with $n$ in box $b_i$. Thus $V_{i+1} \subset \ker \varphi_i$.

Now $\varphi_i(V_i)$ is spanned by the $\varphi_i(e_T)$, where $T_{ii}=n$. The only way that $(T \sigma \pi)_{ii}=n$ for $\sigma \in C_D$, $\pi \in R_D$ is if $\sigma$ and $\pi$ both fix box $b_i$. In other words, we must have $\sigma \in C_{D \backslash \{b_i\}}$ and $\pi \in R_{D \backslash \{b_i\}}$. Thus
\[\varphi_i(e_T) = \varphi_i(T) \ov C(D \backslash \{b_i\}) R(D \backslash \{b_i\}) \in S^{D \backslash \{b_i\}}.\]
Since any injective tableau of shape $D \backslash \{b_i\}$ is of the form $\varphi_i(T)$ for some such $T$, it follows that $\varphi_i \colon V_i \to S^{D \backslash \{b_i\}}$ is surjective.
\end{proof}

\begin{cor} \label{cor-subrep}
Let $B$ be a special transversal of a diagram $D$, and let $n = |D|$. Then $\Res^{\Sym_n}_{\Sym_{n-1}} S^D$ contains an $\Sym_{n-1}$-submodule isomorphic to $\bigoplus_{i=1}^k S^{D \backslash \{b_i\}}$. Using the notation of Proposition~\ref{prop-filtration}, equality holds if and only if $V_1 = S^D$ and $V_{i+1} = \ker \varphi_i$ for all $i$. 
\end{cor}
\begin{proof}
Take the successive quotients in the filtration in Proposition~\ref{prop-filtration} (and use complete reducibility of representations over finite groups).
\end{proof}

For this reason, we make the following definition.
\begin{defn}
A diagram $D$ \emph{branches} with respect to a subset $B \subset D$ if $B$ is a special transversal of $D$, and
\[\Res^{\Sym_n}_{\Sym_{n-1}} S^D \cong \bigoplus_{x \in B} S^{D \backslash \{x\}}. \tag{$*$}\] We call $B$ a \emph{branching set} of $D$.

We say that $D$ \emph{branches completely} with respect to $B$ if $B$ is a branching set of $D$, and $D \backslash \{b\}$ branches completely for all $b \in B$. 
\end{defn}

By Theorems~\ref{thm-northwest} and \ref{thm-forest}, all northwest diagrams and forest diagrams branch completely. Note that the definition of complete branching does not stipulate any relationship between the branching sets for $D$ and $D \backslash \{b\}$.

\begin{rmk} \label{rmk-branching}
\begin{enumerate}[(a)]
\item It is possible for \eqref{eq-branch} to hold for a set $B$ that is not a special transversal. For instance, let \[D = \ydiagram{1+1,2,1}*[*(lightgray)]{0,1,1}\,.\] Then \eqref{eq-branch} holds if $B$ is the set of shaded boxes. However, we will not consider such possibilities here because one does not have the algebraic justification of a filtration as in Proposition~\ref{prop-filtration}. \medskip

\item Not all diagrams have a branching set. For instance, let
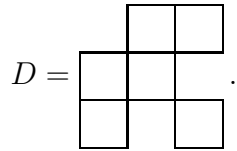
\[D = \ydiagram{1+2,2,1}*{0,0,2+1}\,.\]
A straightforward calculation shows that $S^D \cong S^{33} \oplus 2 \cdot S^{321} \oplus S^{222}$, so \[\Res^{\Sym_6}_{\Sym_5} S^D \cong 3 \cdot S^{32} \oplus 2 \cdot S^{311} \oplus 3 \cdot S^{221}.\]
However, for any box $x \in D$, $S^{D \backslash x} = S^{32} \oplus S^{311} \oplus S^{221}$. Hence no branching rule can exist for $D$.
\medskip

\item Not every branching diagram branches completely with respect to some branching set. For instance, let
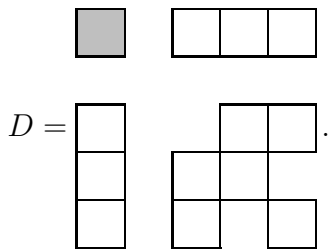
\[D = \ydiagram{2+3,0,1,1,1}*{0,0,3+2,2+2,2+1}*{0,0,0,0,4+1}*[*(lightgray)]{1}\,.\]

By Proposition~\ref{prop-onebox} below, $D$ branches with respect to the shaded box; however, by Proposition~\ref{prop-subdiagram} below, $D$ does not branch completely since it contains as an induced subdigram the nonbranching example from (b) above.
\medskip

\item Not every diagram that branches completely does so with respect to all of its branching sets. Take
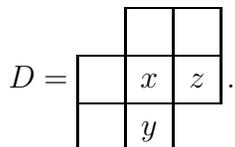
\[D = \yts{\none,\none xz, \none y}*{1+2,3,2}\,.\]
Both $B_1 = \{x\}$ and $B_2 = \{y,z\}$ are branching sets of $D$. However, $D \backslash \{x\}$ does not branch---it is equivalent to the diagram in (b) above---while $D \backslash \{y\}$ and $D \backslash \{z\}$ both branch completely. Hence $D$ branches completely with respect to $B_2$ but not with respect to $B_1$.
\end{enumerate}
\end{rmk}

Proposition~\ref{prop-filtration} also implicitly shows how to construct a basis for $S^D$, assuming one knows that $D$ branches with respect to $B$ and one has a basis for $D \backslash \{x\}$ for each $x \in B$. Indeed, for each $b_i \in B$, let $U_i$ be a set of tableaux of shape $D \backslash \{b_i\}$ such that $\{e_T \mid T \in U_i\}$ forms a basis of $S^{D \backslash \{b_i\}}$. Then $S^D$ has a basis $\{e_T\}$, where $T$ ranges over all tableaux obtained by adding box $b_i$ with label $n$ to some element of $U_i$ for all $b_i \in B$.

In the case that $D$ branches completely with respect to $B$, one can inductively construct this basis, which yields the following proposition.
\begin{prop} \label{prop-basis}
	Choose for each completely branching diagram $D$ a set $B(D)$ with respect to which it branches completely. For each chain $\mathcal C$ of completely branching diagrams \[\varnothing = D^{(0)} \subset D^{(1)} \subset \cdots \subset D^{(n)} = D\] such that $D^{(j)} \backslash D^{(j-1)} = \{x_j\}$ is a single box in $B(D^{(j)})$, let $T_{\mathcal C}$ be the tableau of shape $D$ containing label $j$ in box $x_j$. Then $\{e_{T_{\mathcal C}}\}$ is a basis for $S^D$.
\end{prop}
\begin{proof}
	If $T_{\mathcal C}$ has $n$ in box $b_i \in B(D)$, then by the proof of Proposition~\ref{prop-filtration}, $\varphi_i(e_{T_{\mathcal C}}) = e_{\varphi_i(T_{\mathcal C})}$, where $\varphi_i(T_{\mathcal C})$ is the tableau of shape $D \backslash \{b_i\}$ corresponding to the chain $\mathcal C$ with $D$ removed from the end. Since these images form a basis of $\bigoplus_i S^{D\backslash \{b_i\}}$ (by induction), which is isomorphic to $\bigoplus_i V_i/V_{i+1}$ (since $D$ branches), we have that $\{e_{T_{\mathcal C}}\}$ gives a basis of $S^D$, as desired.
\end{proof}

 Another way of phrasing this is that the restriction of $S^D$ from $\Sym_n$ to $\Sym_{n-1}$, then from $\Sym_{n-1}$ to $\Sym_{n-2}$, and so on, splits $S^D$ into 1-dimensional subspaces, each spanned by some $e_{T_{\mathcal C}}$. (This construction is essentially equivalent to the construction of standard Young tableaux or, for instance, injective balanced labelings for permutation diagrams \cite{EdelmanGreene, FGRS}.)

\subsection{Generalized Garnir relations}
In order to describe the relations for completely branching diagrams, we need to introduce a generalized version of the standard Garnir relations.

For any tableau $A$ of shape $D$ (not necessarily injective), let $\Stab_A \subset \Sym_n$ be the subgroup that stabilizes $A$ (with respect to right action), and define elements of $\CC[\Sym_n]$
\begin{align*}
\Stab(A) &= \sum_{\pi \in \Stab_A} \pi,&
\ov \Stab(A) &= \sum_{\pi \in \Stab_A} \sgn(\pi) \cdot \pi.
\end{align*}

\begin{prop} \label{prop-relations}
Let $A$ be a tableau of shape $D$ such that $A \cdot \sigma$ is not row-strict for any $\sigma \in C_D$. Then
\[\ov \Stab(A) \ov C(D) R(D) = 0,\]
so for any injective tableau $T$,
\[\sum_{\pi \in \Stab_A} \sgn(\pi) \cdot e_{T\pi} = 0.\]
\end{prop}
\begin{proof}
By the condition on $A$, for any $\sigma \in C_D$, there exists a transposition $\tau \in R_D$ that stabilizes $A \cdot \sigma$, so $\sigma \tau \sigma^{-1}$ is a transposition in $\Stab_A$. Then $(\id - \sigma\tau\sigma^{-1})$ is a right factor of $\ov \Stab(A)$ and $(\id + \tau)$ is a left factor of $R(D)$, so since
\[(\id - \sigma\tau\sigma^{-1}) \cdot \sigma \cdot (\id + \tau) = \sigma \cdot(\id - \tau)(\id + \tau) = 0,\]
it follows that $\ov \Stab(A) \cdot \sigma \cdot R(D) = 0$. Summing over all $\sigma \in C_D$ gives the result.
\end{proof}

Note that the usual Garnir relations follow as a special case of Proposition~\ref{prop-relations}: pick two columns of $D$ that intersect fewer than $a$ rows, and let $A$ be a tableau with $a$ boxes in these two columns labeled 1 and all other boxes labeled distinctly. Then $A$ will satisfy the condition in Proposition~\ref{prop-relations} by the pigeonhole principle, and the Garnir relations follow.

It will turn out that Proposition~\ref{prop-relations} suffices to describe all the relations in any Specht module that branches completely. We will identify a special subclass of these relations that will also suffice, which we refer to as the \emph{generalized Garnir relations}. These will be the analogues of the two-column Garnir relations for multiple columns.


\begin{defn}
For $m \geq 1$, choose distinct columns $j_1, j_2, \dots, j_{m+1}$ of the diagram $D$, and pick $m$ disjoint sets of boxes $A = (A_1, A_2, \dots, A_m)$ satisfying:
\begin{itemize}
\item for $i=1, \dots, m$, $A_i$ is a subset of columns $j_i$ and $j_{i+1}$; and
\item $\sum_{i=1}^m |A_i| > c_{j_1} + \dots + c_{j_{m+1}} - c$, where $c_j$ is the number of boxes in column $j$, and $c$ is the number of rows with boxes in all of the columns $j_1, \dots, j_{m+1}$.
\end{itemize}

Let $\Stab_A$ be the subgroup of $\Sym_n$ that stabilizes each $A_i$, and define $\ov \Stab(A)$ as before. Then a \emph{(generalized) Garnir relation} with respect to $A$ is a relation of the form \[T \cdot \ov \Stab(A) \ov C(D) R(D) = 0.\] 
\end{defn}
(Note that each column $j_1, \dots, j_{m+1}$ must contain at least one box of $A$ for the second condition to hold.)

To see that this is indeed a relation, consider the tableau $\widetilde{A}$ of shape $D$ in which the boxes in $A_i$ are labeled $i$, and all other boxes are labeled distinctly. Then $\widetilde{A}$ satisfies the condition of Proposition~\ref{prop-relations}: by the pigeonhole principle, any tableau of the form $\widetilde{A}\sigma$ for $\sigma \in C_D$ must contain some row with $m+1$ boxes using labels $1, \dots, m$, so it cannot be row-strict. We therefore get a relation of $S^D$ by applying Proposition~\ref{prop-relations}.

\begin{ex}
Let $\widetilde{A}$ be the following tableau of shape $D$:
\[\yts{1234,1233,1,\none 1,\none\none 2, \none\none\none 3}
*[*(red)]{1,1,1,1+1} *[*(green)]{1+1,1+1,0,0,2+1} *[*(blue)]{2+1,2+2,0,0,0,3+1}.\]
Let $A_i$ be the set of boxes labeled $i$ for $i=1,2,3$. These $A_i$ satisfy the conditions listed above, so $\ov \Stab(A) \ov C(D) R(D)$ is a four-column Garnir relation.
\end{ex}

\begin{rmk}
	Note that as defined, given $A=(A_1, A_2, \dots, A_m)$ yielding a Garnir relation, it is possible that some subsequence $A' = (A_i, A_{i+1}, \dots, A_j)$ also yields a Garnir relation. In this case, since $\ov\Stab(A)$ is a multiple of $\ov\Stab(A')$, the latter relation implies the former. Therefore one really only needs to consider \emph{minimal} Garnir relations, that is, ones in which $A$ contains no such subsequence. However, it will be useful for us to allow nonminimal Garnir relations in our discussion below, so we will usually not make this distinction.
\end{rmk}


We will need the following result about how one can rearrange terms in a generalized Garnir relations.
\begin{prop} \label{prop-move}
	Let $D$ be a diagram with $n$ boxes with a generalized Garnir relation with respect to $A=(A_1, \dots, A_m)$. Choose $x \in A_p$, and write $A_x = (A_1, \dots, A_p \backslash \{x\}, \dots, A_m)$. Then for any $y \in A_q$, there exists a generalized Garnir relation with respect to $A' = (A_1', \dots, A_m')$ with $\bigcup A_i = \bigcup A_i'$ such that $\ov \Stab(A_x)\ov C(D)R(D)$ is a scalar multiple of $\pi \cdot \ov \Stab(A'_y)\ov C(D)R(D)$ for some $\pi \in \Sym_n$ with $\pi(y)=x$.
\end{prop}
\begin{proof}
	Note that if $x$ and $y$ lie in the same column and $p=q$, we can just take $\pi$ to be the transposition switching $x$ and $y$ and $A' = \pi(A)$. Moreover, if $x=y$ and $q=p+1$, we can let
	\[A' = (A_1, \dots, A_p\backslash \{x\}, A_{p+1} \cup \{x\}, \dots, A_m),\]
	and the result will trivially hold since $A_x = A'_x$.
	
	We first prove the case when $p=q$, but $x$ and $y$ lie in different columns. Since $A$ is a Garnir relation, $\ov\Stab(A)C(D)R(D)=0$. If $Z$ is any set of left coset representatives for $\Stab_A\cap C_D$ in $\Stab_A$, then by Remark~\ref{rmk-mcr},
	\[\sum_{\sigma \in Z} {\sgn(\sigma)\cdot \sigma \ov C(D) R(D)} = 0.\]
	Let $Z_x$ be a set of left coset representatives for $\Stab_{A_x} \cap C_D$ in $\Stab_{A_x}$, and similarly define $Z_y$. Then we can take $Z = Z_x \cup (\pi \cdot Z_y)$, where $\pi$ is the transposition switching $x$ and $y$. Therefore
	\[\sum_{\sigma \in Z_x} {\sgn(\sigma)\cdot \sigma \ov C(D) R(D)} = \pi \sum_{\sigma \in Z_y} {\sgn(\sigma)\cdot \sigma \ov C(D) R(D)}.\]
	Using Remark~\ref{rmk-mcr} again proves the result for $A'=A$. Note that this argument also proves that if $A_p$ lies entirely within one column, then $\ov \Stab(A_x)\ov C(D)R(D) = 0$ (since $Z_y$ will be empty).
	
	The general case follows from applying the above cases repeatedly. Assume without loss of generality that $A_i$ is contained in columns $i$ and $i+1$ and that $p<q$. Choose boxes $z_{p} \in A_p, \dots, z_{q-1} \in A_{q-1}$ such that $z_i$ lies in column $i+1$. Then by applying the result for $x, z_{p} \in A_p$, then moving $z_{p}$ from $A_p$ to $A'_{p+1}$ as in the first paragraph above, then applying the result for $z_{p}, z_{p+1} \in A'_{p+1}$, and so forth, we arrive at the desired result for $x$ and $y$. (If one cannot choose such boxes, then $\ov \Stab(A_x)\ov C(D)R(D)$ will vanish as at the end of the last paragraph.) 
\end{proof}

\begin{ex} \label{ex-garnir}
	We illustrate Proposition~\ref{prop-move}. In the tableaux below, if we first think of $\bullet$ as one of the lowercase letters, the corresponding Garnir relation is (writing a tableau $T$ instead of $e_T$ and using Remark~\ref{rmk-mcr} for ease of notation):
	\[\yts{\bullet aC,b,\none D,\none\none E} *[*(orange)]{1}*[*(red)]{1+1,1}*[*(yellow)]{2+1,0,1+1,2+1}
	\,-\, \yts{\bullet aD,b,\none C,\none\none E}  *[*(orange)]{1}*[*(red)]{1+1,1}*[*(yellow)]{2+1,0,1+1,2+1}
	\,-\, \yts{\bullet aC,b,\none E,\none\none D}  *[*(orange)]{1}*[*(red)]{1+1,1}*[*(yellow)]{2+1,0,1+1,2+1}
	\,-\, \yts{\bullet bC,a,\none D,\none\none E}  *[*(orange)]{1}*[*(red)]{1+1,1}*[*(yellow)]{2+1,0,1+1,2+1}
	\,+\, \yts{\bullet bD,a,\none C,\none\none E}  *[*(orange)]{1}*[*(red)]{1+1,1}*[*(yellow)]{2+1,0,1+1,2+1}
	\,+\, \yts{\bullet bC,a,\none E,\none\none D}  *[*(orange)]{1}*[*(red)]{1+1,1}*[*(yellow)]{2+1,0,1+1,2+1}\]
	\[=\, \yts{a\bullet C,b,\none D,\none\none E}  *[*(orange)]{1+1}*[*(red)]{1,1}*[*(yellow)]{2+1,0,1+1,2+1}
	\,-\, \yts{a\bullet D,b,\none C,\none\none E}  *[*(orange)]{1+1}*[*(red)]{1,1}*[*(yellow)]{2+1,0,1+1,2+1}
	\,-\, \yts{a\bullet C,b,\none E,\none\none D}  *[*(orange)]{1+1}*[*(red)]{1,1}*[*(yellow)]{2+1,0,1+1,2+1},\]
	and now thinking of $\bullet$ as one of the uppercase letters, we get that this also equals:
	\[=\, \yts{aC\bullet,b,\none D,\none\none E}  *[*(orange)]{2+1}*[*(red)]{1,1}*[*(yellow)]{1+1,0,1+1,2+1}
	\,-\, \yts{aC\bullet,b,\none E,\none\none D}  *[*(orange)]{2+1}*[*(red)]{1,1}*[*(yellow)]{1+1,0,1+1,2+1}
	\,+\, \yts{aD\bullet,b,\none E,\none\none C}  *[*(orange)]{2+1}*[*(red)]{1,1}*[*(yellow)]{1+1,0,1+1,2+1}.\]
	Hence $\ov \Stab(A_x) \ov C(D) R(D)$ is a multiple of $\pi \cdot \ov \Stab(A_y') \ov C(D) R(D)$, where the terms in $\ov \Stab(A_x)$ (resp. $\ov \Stab(A_y')$) stabilize the sets of boxes containing lowercase and uppercase letters on the left (resp. right) hand side.
\end{ex}

\subsection{Maps between Specht modules}

Let $\psi \colon D \to E$ be a bijection between the boxes of two diagrams. In the same way that we think of tableaux of shapes $D$ and $E$ as elements of $\Sym_n$, we may think of $\psi$ as an element of $\Sym_n$. If $T$ is a tableau of shape $E$, then $T \psi$ is the tableau of shape $D$ whose labels correspond to those of $T$ under $\psi$.

If $D$ and $E$ have rows of the same sizes, then we can use a row-preserving bijection $\psi\colon D \to E$ to identify $R(D)$ and $R(E)$ via $\psi R(D) = R(E)\psi$. Since $S^D \subset \CC[\Sym_n]R(D)$ and $S^E \subset \CC[\Sym_n]R(E)$, we can then ask for the relationship between $S^E\psi$ and $S^D$ inside $\CC[\Sym_n]R(D)$.

First, we give a slightly modified version of Proposition~\ref{prop-relations}.

\begin{prop} \label{prop-unique}
Let $A$ be a tableau of shape $D$ such that $A$ is the unique row-strict tableau of the form $A \cdot \sigma$ for $\sigma \in C_D$. Then $\ov\Stab(A) \ov C(D) R(D)$ is a nonzero scalar multiple of $\ov \Stab(A) R(D)$. In particular, $\ov \Stab(A) R(D) \in S^D$.
\end{prop}
\begin{proof}
From the proof of Proposition~\ref{prop-relations}, if $A \cdot \sigma$ is not row-strict, then $\ov \Stab(A) \sigma R(D)=0$. Hence
\begin{align*}\ov \Stab(A) \ov C(D) R(D) &= \ov \Stab(A) \cdot \sum_{\sigma \in C_D \cap \Stab_A} \sgn(\sigma)\sigma \cdot R(D)\\ &= |C_D \cap \Stab_A| \cdot \ov \Stab(A) R(D). \qedhere
\end{align*}
\end{proof}

From this, we can deduce the following embedding of Specht modules.

\begin{prop} \label{prop-submodule}
Let $A$ satisfy the hypothesis of Proposition~\ref{prop-unique}. Let $E$ be the diagram with boxes $(i, j)$ for all $j$ appearing as labels in row $i$ of $A$. Then the corresponding row-preserving bijection $\psi\colon D \to E$ induces an injection of $\Sym_n$-modules $\psi^*\colon S^E \to S^D$ sending $e_T \mapsto e_T \cdot \psi$.
\end{prop}
\begin{proof}
By Proposition~\ref{prop-unique}, \[T \cdot \ov C(E) R(E) \psi = T \cdot \ov C(E) \psi R(D) = T\psi \cdot \ov \Stab(A) R(D) \in S^D.\qedhere\]
\end{proof}

\begin{ex}
Let $A$ be the tableau of shape $D$ shown to the left below with corresponding diagram $E$ shown to the right.
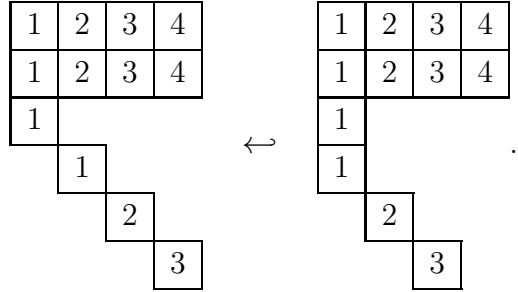
\[\yts{1234,1234,1,\none 1,\none\none 2, \none\none\none 3} \quad \hookleftarrow \quad 
\yts{1234,1234,1,1,\none 2, \none\none 3}.\]
Since $A$ satisfies the condition of Proposition~\ref{prop-unique}, there is an inclusion of $S^E$ into $S^D$.
\end{ex}

The most important example of Proposition~\ref{prop-submodule} is when $E$ is obtained from $D$ by replacing two columns of $D$ with their intersection and their union. More generally, we can do something similar with any $k$ columns.

\begin{prop} \label{prop-smash}
Let $D$ be a diagram, and fix $k\geq 1$. Let $r_i = |\{(i,j) \in D \mid j \leq k\}|$. Let $E$ be the diagram obtained from $D$ by replacing the first $k$ columns with $\{(i,j)\mid j \leq r_i\}$. Then there exists an embedding of $S^E$ into $S^D$ as in Proposition~\ref{prop-submodule}.
\end{prop}
\begin{proof}
Define the tableau $A$ of shape $D$ by $A_{ij} = j$ if $j>k$, and otherwise $A_{ij} = |\{(i,j') \in D \mid j' \leq j\}|$. (The labels in the first $k$ columns of $A$ increase consecutively in each row starting with 1.)

Then $A$ satisfies the condition of Proposition~\ref{prop-unique}: if $A\sigma \neq A$ for some $\sigma \in C_D$, then let $j$ be the leftmost column in which they differ. Then there exists some row $i$ such that $(A\sigma)_{ij} < A_{ij}$. We must have $j \leq k$ since labels after the first $k$ columns cannot change, but then $A\sigma$ will not be row-strict since $(A\sigma)_{ij}$ will repeat a label earlier in the column.
\end{proof}

\begin{ex}
Let $D$ be the diagram shown to the left below. Then applying Proposition~\ref{prop-smash} with $k=4$, we find that there is an inclusion of $S^E$ into $S^D$ induced by the row-preserving bijection mapping these two tableaux to each other.
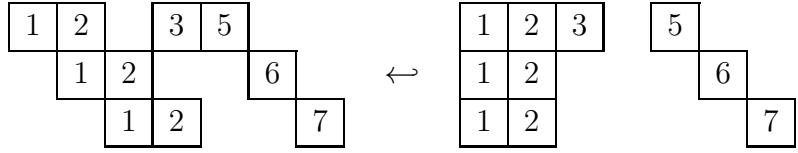
\[
\yts{12\none35,\none12\none\none6,\none\none12\none\none7} \quad \hookleftarrow \quad
\yts{123\none5,12\none\none\none6,12\none\none\none\none7}
\]
\end{ex}

Clearly we can apply Proposition~\ref{prop-smash} to modify any $k$ columns, not just the first $k$.

Propositions~\ref{prop-relations}, \ref{prop-unique}, and \ref{prop-submodule} have dual versions concerning column-strictness.
\begin{prop}\label{prop-mergerows}
Let $A$ be a tableau of shape $D$.
\begin{enumerate}[(a)]
\item Suppose $A\sigma$ is not column-strict for any $\sigma \in R_D$. Then $\ov C(D)R(D)\Stab(A) = 0$.
\item Suppose $A$ is the unique column-strict tableau of the form $A\sigma$ for $\sigma \in R_D$. Then $\ov C(D)R(D)\Stab(A)$ is a nonzero scalar multiple of $\ov C(D)\Stab(A)$.
\item Let $A$ be as in part (b), and let $E$ be the diagram with boxes $(i, j)$ for all $i$ appearing as labels in column $j$ of $A$. Then the corresponding column-preserving bijection $\psi \colon D \to E$ induces a surjection $\psi_*\colon S^D \to S^E$ sending $e_T \mapsto e_{T\psi^{-1}}$.
\end{enumerate}
\end{prop}
\begin{proof}
If $A\sigma$ is not column-strict for some $\sigma \in R_D$, then there exists a transposition $\tau \in C_D$ such that $\sigma\tau\sigma^{-1} \in \Stab_A$. Then $(\id - \tau)\sigma^{-1}(\id + \sigma\tau\sigma^{-1})=0$ implies $\ov C(D)\sigma^{-1}\Stab(A) = 0$ as in Proposition~\ref{prop-relations}. Part (a) follows by summing over all $\sigma \in R_D$.

For part (b), summing over all $\sigma \in R_D$ gives
\begin{align*}
\ov C(D)R(D)\Stab(A) &= \sum_{\sigma \in R_D \cap \Stab(A)} \ov C(D) \sigma^{-1} \Stab(A) \\ &= |R_D \cap \Stab(A)| \cdot \ov C(D)\Stab(A).
\end{align*}

For part (c), note that $\psi_*$ is, up to a scalar factor, just multiplication by $\Stab(A) \psi^{-1}$, since for any tableau $T$ of shape $D$, $e_T\cdot \Stab(A)\psi^{-1}$ is a constant times
\[T \cdot \ov C(D)\Stab(A)\psi^{-1} = T \cdot \ov C(D) \psi^{-1} R(E) = T\psi^{-1} \cdot \ov C(E)R(E). \qedhere\]
\end{proof}




\section{Branching}
In this section, we will use the results of \S3 to give a combinatorial criterion for when a diagram branches completely.

\subsection{Maximal rectangles}

The key condition we will need for complete branching involves \emph{maximal rectangles}.

\begin{defn}
A \emph{rectangle} $P \subset D$ is a subset of the form $P_1 \times P_2$ for some $P_1, P_2 \subset \NN$. A \emph{maximal rectangle} is one that is maximal under inclusion.
\end{defn}

For example, in the Young diagram of a partition, there is one maximal rectangle for each corner box of the diagram.

\begin{defn}
We say that $B \subset D$ is an \emph{exact hitting set} for the maximal rectangles of $D$ if every maximal rectangle contains exactly one element of $B$.
\end{defn}

The corner boxes of a Young diagram form such an exact hitting set. This is not a coincidence, as we shall see from the following result.

\begin{prop} \label{prop-hitting}
Suppose $B$ is a branching set of $D$. Then $B$ is an exact hitting set for the maximal rectangles of $D$.
\end{prop}
\begin{proof}
Since $B = \{b_1, b_2, \dots, b_k\}$ is a branching set, it is a special transversal. Thus no rectangle can contain two elements of $B$. It therefore suffices to show that every maximal rectangle contains an element of $B$. We will assume without loss of generality that $b_i = (i, i)$.

Let $P = P_1 \times P_2$ be a maximal rectangle, and let $p = \min(P_1)$; if $p>k$, we may assume without loss of generality that $p=k+1$. Suppose that $b_p \not \in P$. Using the notation of Proposition~\ref{prop-filtration}, we will show that $\ker \varphi_{p-1} \not \subset V_p$, which will contradict that $B$ is a branching set. (If $p=1$, we will show that $S^D \not \subset V_p$.)

Let $E$ be the diagram obtained from $D$ by applying Proposition~\ref{prop-smash} to the columns in $P_2$. Since $P$ is a maximal rectangle, there is a column $j$ of $E$ that contains boxes precisely in the rows of $P_1$. Let $\psi\colon D\to E$ be the bijection found in Proposition~\ref{prop-smash}. Let $W \subset S^E$ be the linear subspace spanned by $e_T$, where $T$ is a tableau of shape $E$ with $T_{pj}=|D|=n$. Then $\psi^*(W)$ lies in the kernel of $\varphi_{p-1}$ since by our choice of $p$, $e_T\psi$ does not contain any terms with $n$ in rows $1, \dots, p-1$. If $p=k+1$, it follows immediately that $\ker \varphi_{p-1}\not\subset V_{p} = 0$. Otherwise, if $D' = D \backslash \{b_p\}$ and $E' = E \backslash \{(p, j)\}$, then $\varphi_p(V_p) = S^{D'}$ and $\varphi_p(\psi^*(W)) = \widetilde{\psi}^*(S^{E'}),$ where $\widetilde{\psi}\colon D' \to E'$ is a certain row-preserving bijection. We will show that $\psi^*(W) \not \subset V_p$ by showing that $\widetilde{\psi}^*(S^{E'}) \not \subset S^{D'}$.

Since $b_p \not \in P$ implies $p \not\in P_2$, by maximality of $P$ there exists a row $i \in P_1$ such that $(i,p) \not \in D$. Any row in $P_1$ is identical in $D$ and $E$, so rows $i$ and $p$ of $E'$ are identical to those of $D'$ except $(p,j)$ has been replaced by $b_p=(p,p)$. Since $(i,j)$ lies in both $D'$ and $E'$ but $(i,p)$ lies in neither, this means that rows $p$ and $i$ intersect fewer columns in $D'$ than in $E'$.
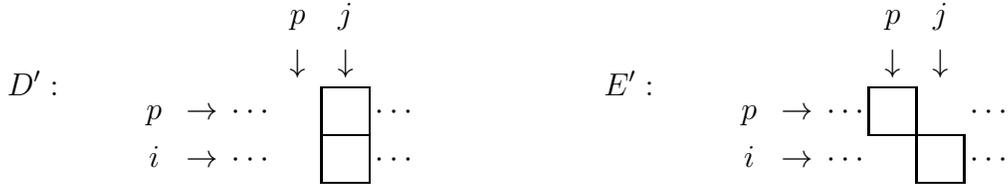
\[
D': \qquad
\by
\none&\none&\none&\none[p]&\none[j]\\
\none&\none&\none&\none[\downarrow]&\none[\downarrow]\\
\none[p]&\none[\to]&\none[\cdots]&\none&&\none[\cdots]\\
\none[i]&\none[\to]&\none[\cdots]&\none&&\none[\cdots]\\
\ey
\qquad\qquad\qquad
E': \qquad
\by
\none&\none&\none&\none[p]&\none[j]\\
\none&\none&\none&\none[\downarrow]&\none[\downarrow]\\
\none[p]&\none[\to]&\none[\cdots]&&\none&\none[\cdots]\\
\none[i]&\none[\to]&\none[\cdots]&\none&&\none[\cdots]\\
\ey
\]

\bigskip 
Let $A$ be the tableau of shape $E'$ for which $A_{pq} = 1$ for all $q$, $A_{iq} = 1$ if $(p,q) \not \in E'$, $A_{iq}=2$ otherwise, and all other boxes are labeled by their row. Then $A$ satisfies the condition of Proposition~\ref{prop-mergerows}(b), so $S^{E'}\Stab(A)$ does not vanish. But
\[S^{E'}\Stab(A) = \widetilde{\psi}^*(S^{E'}) \cdot \widetilde{\psi}^{-1}\Stab(A) = \widetilde{\psi}^*(S^{E'}) \cdot \Stab(A\widetilde\psi)\cdot\widetilde\psi^{-1},\]
and $A\widetilde\psi$ satisfies the condition of Proposition~\ref{prop-mergerows}(a)---$D'$ has one fewer column intersecting rows $p$ and $i$ than $E'$, so $(A\widetilde\psi)\sigma$ cannot be column-strict for any $\sigma \in R_{D'}$. Hence $S^{D'}\Stab(A\widetilde\psi)=0$, so $S^{D'}$ is annihilated by $\Stab(A\widetilde\psi)$ while $\widetilde{\psi}^*(S^{E'})$ is not. Thus $\widetilde{\psi}^*(S^{E'}) \not \subset S^{D'}$, which completes the proof.
\end{proof}

\begin{rmk}
	The converse of Proposition~\ref{prop-hitting} is not true, that is, if $B$ is an exact hitting set for the maximal rectangles of $D$, then $B$ may not be a branching set. For one, $B$ may not be a special transversal: for example, it is easy to construct an exact hitting set for the maximal rectangles of the diagram in Remark~\ref{rmk-branching}(b) by taking exactly one box in each row and column, but such a set $B$ is evidently not a special transversal.
	
	Even if $B$ is a special transversal, it need not be a branching set. For example, let
	\[D = \ydiagram{1+2,2,1}*[*(lightgray)]{5+1,4+1,2+1,0,1+1,1}\,.\]
	The set $B$ of shaded boxes is a special transversal and an exact hitting set for the maximal rectangles of $D$. However, an explicit computation shows that $D$ does not branch with respect to $B$: the multiplicity of $S^{51111}$ in $\Res^{\Sym_{10}}_{\Sym_9} S^D$ is 2, while its multiplicity in $\bigoplus_{x \in B} S^{D \backslash \{x\}}$ is only 1.
\end{rmk}

Even though the converse of Proposition~\ref{prop-hitting} is false, we will give a partial converse by restricting to completely branching diagrams in our main theorem below, Theorem~\ref{thm-main}. For now, we will settle for the following weaker statement.

\begin{prop} \label{prop-v1}
	Let $D$ be a diagram with $n$ boxes, and let $B \subset D$ be a special transversal. If $B$ is an exact hitting set for the maximal rectangles of $D$, then $S^D$ is spanned by $e_T$, where $T_b = n$ for some $b \in B$.
\end{prop}
In other words, in the notation of Proposition~\ref{prop-filtration}, $V_1 = S^D$.
\begin{proof}
	Assume without loss of generality that $B=\{b_1, \dots, b_k\}$ with $b_i = (i,i)$. Let $V_1$ be the span of $e_T$ for which $T_{ii} = n$ for some $i \leq k$, as in Proposition~\ref{prop-filtration}. We need to show that $e_T \in V_1$ for any injective tableau $T$. Clearly this holds if $T$ contains $n$ in one of the columns $1, \dots, k$ (using a one-column relation).
	
	Suppose then that $n$ lies in column $j>k$. The boxes in column $j$ form a rectangle, so they lie in some maximal rectangle $P$. If $B$ is an exact hitting set, then $P$ contains some element of $B$, say $b_i$. Then for any box in column $j$, there is a box in the same row in column $i$. Let $A$ be the set of boxes in column $i$ together with the box containing $n$. Then there is a two-column Garnir relation for the set $A$ that expresses $e_T$ as a linear combination of $e_{T'}$ where $n$ lies in column $i$ of $T'$ (or see Proposition~\ref{prop-move}). Thus $e_T \in V_1$.
\end{proof}

One consequence of Propositions~\ref{prop-hitting} and \ref{prop-v1} is that they can be used to give a necessary and sufficient condition for when a diagram $D$ branches off a single box.
\begin{prop} \label{prop-onebox}
	Let $D$ be a diagram and $x=(i,j) \in D$ any box. Then $\{x\}$ is a branching set for D if and only if every maximal rectangle of $D$ contains $x$. Equivalently, for any box $(i',j') \in D$, both $(i,j')$ and $(i',j)$ must lie in $D$.  
\end{prop}
\begin{proof}
	One direction follows immediately from Proposition~\ref{prop-hitting}.	For the other direction, suppose every maximal rectangle of $D$ contains $x$. By Proposition~\ref{prop-v1}, $S^D$ is spanned by $e_T$ where $T_{x}=n$. Given such a tableau $T$, let $\varphi(T)$ be the tableau of shape $D \backslash \{x\}$ obtained by removing box $x$. By Proposition~\ref{prop-filtration}, we need to show that the map $\varphi\colon S^D \to S^{D \backslash \{x\}}$ that sends $e_T \mapsto e_{\varphi(T)}$ is injective.
	
	For any box $y = (i',j) \neq x$ in column $j$ of $D$, note that $(i',j')\in D$ implies $(i,j')\in D$. Hence if $A$ is the tableau of shape $D \backslash \{x\}$ for which $A_y = i$ and all other boxes are labeled by their row, then $A$ is the unique column-strict tableau of the form $A\sigma$ for $\sigma \in R_{D \backslash \{x\}}$. Let $\psi_y \colon D \to D$ be the transposition switching $x$ and $y$. Then $\psi_y$ restricts to a bijection $\overline \psi_y\colon D \backslash \{x\} \to D \backslash \{y\}$ that, by Proposition~\ref{prop-mergerows}, induces a surjection $(\overline\psi_y)_*\colon S^{D \backslash \{x\}} \twoheadrightarrow S^{D \backslash \{y\}}$. 
	
	Now suppose $\sum_T c_Te_T \in \ker \varphi$, so that
	\[0 = \varphi\left(\sum_T c_Te_T\right) = \sum_T c_T e_{\varphi(T)}=\sum_T c_T \varphi(T) \cdot \ov C(D \backslash \{x\})R(D \backslash \{x\}) \in S^{D \backslash \{x\}}.\]
	Inserting box $x$ with label $n$ back into each of these tableaux, we find that \[\sum_T c_T T \cdot \ov C(D \backslash \{x\}) R(D \backslash \{x\}) = 0. \tag{$\dagger$}\]
	Similarly,
	\[0 = (\ov \psi_y)_*\left(\sum_T c_Te_{\varphi(T)}\right) = \sum_T c_T e_{\varphi(T)\ov\psi_y^{-1}} = \sum_T c_T (\varphi(T) \ov\psi_y^{-1}) \ov C(D \backslash \{y\})R(D \backslash \{y\}) \in S^{D \backslash \{y\}}.\]
	Inserting box $y$ with the label $n$ into all the terms on the right, we must also get 0, so
	\[\sum_T c_T(T \psi_y^{-1}) \ov C(D \backslash \{y\}R(D \backslash \{y\}) = 0\tag{$\ddagger$}\]
	Since \[\ov C(D) = \ov C(D \backslash \{x\}) - \sum_y \psi_y^{-1} \cdot \ov C(D \backslash \{y\}),\]
	\begin{align*}
		\sum_T c_Te_T &= \sum_T c_T T\cdot \ov C(D)R(D)\\
		& = \sum_T c_T T \cdot \ov C(D \backslash \{x\})R(D) - \sum_T \sum_y c_T(T \psi_y^{-1})\ov C(D \backslash \{y\})R(D).
	\end{align*}
	But $R(D \backslash \{x\})$ and $R(D \backslash \{y\})$ are both left factors of $R(D)$, so equations $(\dagger)$ and $(\ddagger)$ imply $\sum_T c_Te_T = 0$, as desired.
\end{proof}
One can also prove Proposition~\ref{prop-onebox} using the box-complementation symmetry of Specht modules shown by Magyar \cite{Magyar}.

\subsection{Complete branching}
We are now ready to prove our main theorem giving a criterion for when a diagram branches completely.
\begin{thm} \label{thm-main}
Let $D$ be a diagram, and let $B \subset D$ be a special transversal.
\begin{enumerate}[(a)]
\item The diagram $D$ branches completely with respect to $B$ if and only if $D \backslash \{b\}$ branches completely for all $b \in B$, and $B$ is an exact hitting set for the maximal rectangles of $D$.
\item If $D$ branches completely, then the generalized Garnir relations generate all the relations in $S^D$ (together with the one-column relations).
\end{enumerate}
\end{thm}
\begin{proof}
We prove both statements by induction on the number of boxes $n=|D|$. One direction of part (a) follows from Proposition~\ref{prop-hitting}, so suppose that $B$ is a special transversal as well as an exact hitting set for the maximal rectangles of $D$ and that $D\backslash \{b\}$ branches completely for all $b \in B$. We need to show that $D$ branches with respect to $B$ and that the relations in $D$ are generated by generalized Garnir relations. 

Assume without loss of generality that $B = \{b_1, \dots, b_k\}$ with $b_i = (i,i)$. Using the notation of Proposition~\ref{prop-filtration}, we need to show that $V_1 = S^D$ and $V_{i+1} = \ker \varphi_i$. The first equality holds from Proposition~\ref{prop-v1}.

Suppose $\sum_T c_Te_T \in \ker \varphi_i$, where $T$ ranges over tableaux with $T_{ii}=n$. We need to show that $\sum_T c_Te_T \in V_{i+1}$. By the inductive hypothesis, the relations in $S^{D \backslash \{b_i\}}$ are generated by the generalized Garnir relations, so it suffices to prove the claim for $T \cdot \ov \Stab(A) \ov C(D) R(D)$, where $T$ is a tableau of shape $D$ with $T_{ii}=n$, and $A=(A_1, \dots, A_m)$ gives a generalized Garnir relation in $D \backslash \{b_i\}$. Let $j_1, \dots, j_{m+1}$ be the columns involved in this relation, and let column $j$ have $c_j$ boxes in $D \backslash \{b_i\}$ and $c_j'$ boxes in $D$. (Hence $c_j' = c_j$ for $j \neq i$, and $c_i' = c_i+1$.) Let $c$ be the number of rows of $D \backslash \{b_i\}$ containing boxes in all of columns $j_1, \dots, j_{m+1}$, and similarly define $c'$ for $D$.

If $A$ also defines a Garnir relation in $D$, then $\ov \Stab(A) \ov C(D) R(D)$ will vanish, so we will be done. Since we know that $A$ defines a Garnir relation in $D \backslash \{b_i\}$, we have
\[\sum_{q=1}^m |A_q| > \sum_{p=1}^{m+1} c_{j_p} - c.\]
When we pass from $D\backslash \{b_i\}$ to $D$, the right hand side of the inequality above can only increase by 1, so $A$ will not define a Garnir relation in $D$ if and only if $j_p=i$ for some $p$, $c=c'$, and
\[\sum_{q=1}^m |A_q| = \sum_{p=1}^{m+1} c'_{j_p} - c,\]
so let us assume this is the case.

A similar argument shows that $A\cup\{b_i\}=(A_1, \dots, A_p \cup \{b_i\}, \dots, A_m)$ will always give a Garnir relation in $D$. If $A$ contains some box in column $i'$ where $i<i'\leq k$, then applying Proposition~\ref{prop-move} to $A\cup\{b_i\}$ gives that $T \cdot \ov \Stab(A)\ov C(D)R(D)$ either vanishes or is a multiple of $T\pi \cdot \ov \Stab(A') \ov C(D)R(D)$ for some $A'$ and some $\pi$, where $\pi^{-1}$ maps $b_i$ into column $i'$. But then $T\pi$ is a tableau with the label $n$ in column $i'$, so $T\pi \cdot \ov \Stab(A')\ov C(D)R(D) \in V_{i+1}$, as desired.

Therefore, we may assume that $A$ does not contain any boxes in columns $i+1, \dots, k$. Choose any box $y$ of $A$ in column $j_{m+1}$, and use Proposition~\ref{prop-move} to write $T \cdot \ov \Stab(A) \ov C(D)R(D)$ as a scalar multiple of $T\pi \cdot \ov \Stab(A') \ov C(D)R(D)$, where $T\pi$ is a tableau with $n$ in column $j_{m+1}$ and $(\bigcup A) \cup \{b_i\} = (\bigcup A') \cup \{y\}$.

Consider the $c$ rows that intersect all of the columns $j_1, \dots, j_{m+1}$. The intersection of these rows and columns forms a rectangle which must be contained in some maximal rectangle $P$ in $D$. Since $P$ is a maximal rectangle, it must contain an element of $B$. But since $c=c'$, $b_i \not \in P$, so some other element $b_{i'}$ must lie in $P$. We cannot have $i'<i$, for then $(i',i) = (i', j_p)$ could not lie in $P$ by the definition of a special transversal. Hence we must have $i'>i$. Now let $C_{i'}$ be the set of boxes in column $i'$, and consider $A'' = (A_1', \dots, A_m', C_{i'} \cup \{y\})$. By our choice of $i'$, there are still exactly $c$ rows that intersect all of the columns $j_1, \dots, j_{m+1}, i'$.
Then
\[\sum_{q=1}^{m+1} |A_q''| = \sum_{q=1}^m |A_q| + |C_{i'}|+1 = \sum_{p=1}^{m+1} c_{j_p}'-c+c'_{i'}+1 > \sum_{p=1}^{m+1} c_{j_p}'+c'_{i'}-c.\]
Thus $A''$ gives a generalized Garnir relation for $D$. Applying Proposition~\ref{prop-move} to $A''$ using $y$ and $b_{i'}$ gives
\[(T\pi) \cdot \ov \Stab(A'' \backslash \{y\}) \ov C(D)R(D) \in V_{i'} \subset V_{i+1}.\]
Since $A'' \backslash \{y\} = (A_1' \dots, A_m', C_{i'})$ and $\Stab_{C_{i'}} \subset C_D$, it follows that $\ov \Stab(A'' \backslash \{y\}) \ov C(D)$ is just a scalar multiple of $\ov \Stab(A') \ov C(D)$. Hence $(T\pi) \cdot \ov \Stab(A') \ov C(D) R(D) \in V_{i+1}$, and therefore $T \cdot \ov \Stab(A) \ov C(D)R(D) \in V_{i+1}$, as desired.

Since Proposition~\ref{prop-move} uses only one-column and generalized Garnir relations, these relations generate all the relations in $S^D$.
\end{proof}

Note that we have shown something slightly stronger, namely that the converse of Proposition~\ref{prop-hitting} is true for a special transversal $B$ as long as all the relations in $S^{D \backslash \{x\}}$ are generated by generalized Garnir relations for all $x \in B$.

Theorem~\ref{thm-main} can be used to iteratively construct all completely branching diagrams: if all completely branching diagrams with $n$ boxes are known, then one can check whether a diagram with $n+1$ boxes branches completely by checking for the existence of a branching set using Theorem~\ref{thm-main}(a) without using any algebraic structure.

\subsection{Straightening}

The proof of Theorem~\ref{thm-main} (together with the proof of Proposition~\ref{prop-move}) implicitly gives a straightening rule for any diagram $D$ that branches completely. More precisely, suppose that we wish to write some $e_T$ in terms of the basis $\{e_{T_{\mathcal C}}\}$ as described in Proposition~\ref{prop-basis}, and suppose that $T$ contains the label $n$ in box $x$. Theorem~\ref{thm-main} then tells us to perform the following procedure to straighten it:
\begin{itemize}
	\item If $x \not\in B(D) = \{b_1, \dots, b_k\}$, then as in Proposition~\ref{prop-v1} there exists some $i$ such that the column containing $x$ (as a subset of $\mathbf N$) is a subset of the column containing $b_i$. Then there is a two-column (or one-column) Garnir relation writing $e_T$ in terms of $e_{T'}$ where $T'$ contains the label $n$ in box $b_i$. Thus we may assume $x \in B(D)$.
	\item If $x = b_i$, then write $\varphi_i(T)$ for the tableau of shape $D \backslash \{b_i\}$ obtained by removing the box $b_i$ containing $n$ from $T$. By induction, $e_{\varphi_i(T)}$ can be written as a linear combination of $e_{\varphi_i(T_{\mathcal C})}$ in $S^{D \backslash \{b_i\}}$ using $(m+1)$-column generalized Garnir relations of the form $\ov\Stab(A) \ov C(D \backslash \{b_i\})R(D \backslash \{b_i\})$. Hence $e_T$ differs from a linear combination of $e_{T_\mathcal C}$ by terms of the form $\ov\Stab(A) \ov C(D)R(D)$, so it suffices to straighten expressions of this form.
	\item If $A$ also defines a generalized Garnir relation in $D$, then $\ov\Stab(A) \ov C(D)R(D)$ vanishes. Otherwise, either $A$ involves a column containing $b_{i'}$ with $i'>i$ or it doesn't. 
	\item If it does, then Proposition~\ref{prop-move} writes $\ov\Stab(A) \ov C(D)R(D)$ in terms of tableaux containing $n$ in box $b_{i'}$ using $(m+1)$-column Garnir relations.
	\item Otherwise, the maximal rectangle intersecting all the columns of $A$ contains some $b_{i'}$ with $i'>i$, and then there exists an $(m+2)$-column Garnir relation to write $\ov\Stab(A) \ov C(D)R(D)$ in terms of tableaux containing $n$ in $b_{i'}$.
\end{itemize}
Note that the label $n$ always moves from $b_i$ to $b_{i'}$ with $i'>i$, so this straightening procedure always terminates.

\begin{ex}
	We give an example of one step of the straightening procedure. Consider the following diagram $D$ with branching set $B = \{b_1,b_2,b_3\}$ as indicated. 
	\[D = \yts{\none\none {b_3},{b_1},\none {b_2}}*{3,1,1+1}*[*(lightgray)]{2+1,1,1+1}.\]
	Assume that we have chosen for $S^{D \backslash \{b_2\}}$ the usual basis $\{e_T\}$ where $T$ is a standard Young tableau (with increasing rows and columns from left to right and top to bottom). Consider the element $e_T$, where \[T = \yts{214,3,\none 5}\,.\]
	If we were to remove box $b_2$ containing 5, then the resulting tableau can be straightened in $S^{D \backslash \{b_2\}}$ using a Garnir relation as follows (using $T$ to represent $e_T$ for ease of notation):
	\[\yts{214,3}*[*(yellow)]{2,1} = \yts{124,3}*[*(yellow)]{2,1}-\yts{134,2}*[*(yellow)]{2,1}\]
	If we add back in box $b_2$, the result is not a relation in $S^D$. However, the maximal rectangle intersecting the two columns involved in this relation contains $b_3$. Hence	we can use a 3-column Garnir relation to write
	\[\yts{214,3,\none5}*[*(yellow)]{2,1}*[*(orange)]{2+1,0,1+1} = \yts{124,3,\none5}*[*(yellow)]{2,1}*[*(orange)]{2+1,0,1+1}
	-\yts{134,2,\none5}*[*(yellow)]{2,1}*[*(orange)]{2+1,0,1+1}
	+\yts{215,3,\none4}*[*(yellow)]{2,1}*[*(orange)]{2+1,0,1+1} -\yts{125,3,\none4}*[*(yellow)]{2,1}*[*(orange)]{2+1,0,1+1}
	+\yts{135,2,\none4}*[*(yellow)]{2,1}*[*(orange)]{2+1,0,1+1}\,.\]
	The first two terms on the right hand side are now basis elements, while the last three terms have the largest label in $b_3$ rather than $b_2$.
	\end{ex}

\subsection{Relation to previous results}

Using Theorem~\ref{thm-main}, we can give simple, combinatorial proofs that northwest diagrams and forest diagrams branch completely. The two corollaries below are restatements of Theorem~\ref{thm-northwest} and Theorem~\ref{thm-forest}.

Recall that it was shown in \cite{ReinerShimozono2} that every northwest diagram can have its rows rearranged to be in initial segment order.
\begin{cor}
	Let $D$ be a northwest diagram with its rows arranged in initial segment order. Let $B$ be the set of boxes $x$ such that $x$ is bottommost in its column, and no box $y$ that is bottommost in its column lies to the left of $x$ (in the same row). Then $D$ branches completely with respect to $B$.
\end{cor}

\begin{proof}
	Since any $x \in B$ is bottommost in its column, removing it will preserve the northwest property. Therefore by induction and Theorem~\ref{thm-main}, it suffices to check that $B$ is a special transversal and every maximal rectangle of $D$ intersects $B$. Denote the boxes in $B$ by $b_1, \dots, b_k$ from bottom to top. Since each box of $B$ lies in a different row and each is bottommost in its column, this ordering shows that $B$ is a special transversal.
	
	Consider any maximal rectangle $P = P_1 \times P_2$. Note that $D$ cannot contain any box $(i,j)$ with $i \geq \max(P_1)$ and $j < \max (P_2)$ unless $j \in P_2$, for otherwise using the northwest property with $(i,j)$ and each box in the last column of $P$ would imply that we could add $j$ to $P_2$, contradicting maximality.
	
	For any column $j$, let $x_j$ be the bottommost box in column $j$. Among all $x_j$ for $j \in P_2$, let $y = x_c = (r,c)$ be the one with $r$ minimum and, among those, the one with $c$ minimum. By the previous paragraph, $y$ cannot have any other $x_j$ directly to its left, so $y \in B$.
	
	In fact, $y$ also lies in $P$: suppose the rightmost box in row $r$ lies in column $d$. By above, either $d\geq \max (P_2)$ or $d \in P_2$. If $d \geq \max (P_2)$, then using the northwest property with all $x_j$ for $j \in P_2$ implies that we could add row $r$ to $P_1$, contradicting maximality of $P$. If instead $d \in P_2$, then the northwest property with all $x_j$ for $j \in P_2$ and $j < d$ implies that row $r$ is an initial segment of row $\max (P_1)$, contradicting initial segment order. It follows that $y \in B \cap P$, as desired.
\end{proof}

Recall that it was shown in \cite{Liu} that every forest has an almost perfect matching.
\begin{cor}
	Let $D$ be a forest diagram, and let $B$ be the set of boxes corresponding to any almost perfect matching of $G(D)$. Then $D$ branches completely with respect to $B$.
\end{cor}
\begin{proof}
	By Proposition~\ref{prop-alternating}, $B$ is a special transversal since $G(D)$ has no cycles. Maximal rectangles in $D$ correspond to maximal bicliques (induced complete bipartite subgraphs) of $G(D)$. In a forest, the maximal bicliques are either isolated edges or the set of edges incident to a vertex of degree greater than 1. Thus, exact hitting sets correspond precisely to almost perfect matchings. Since removing any box from a forest diagram results in a forest diagram, the result follows from Theorem~\ref{thm-main} by induction.
\end{proof}

\subsection{The class of completely branching diagrams}

Let $\mathcal B$ be the class of completely branching diagrams. By Theorem~\ref{thm-main}, $D \in \mathcal B$ if and only if there exists a special transversal $B$ that is an exact hitting set for the maximal rectangles of $D$, and $D \backslash \{x\} \in \mathcal B$ for all $x \in B$. We will now discuss some properties of $\mathcal B$.

First, we show that $\mathcal B$ is closed under taking \emph{induced subdiagrams}, that is, diagrams formed by restricting to a smaller set of rows and columns (or equivalently, by intersecting with a rectangle).
\begin{prop} \label{prop-subdiagram}
	Let $D \in \mathcal B$ be a completely branching diagram, and let $E \subset D$ be an induced subdiagram. Then $E \in \mathcal B$.
\end{prop}
\begin{proof}
	For the sake of contradiction, choose $D \in \mathcal B$ to be minimal such that there exists some subdiagram $E \subset D$ with $E \not\in \mathcal B$. Since $D$ branches completely, it has some branching set $B$. Then we must have $B \subset E$, for if $x \in B$ but $x \not \in E$, then $D \backslash \{x\} \in \mathcal B$ contains $E$ as an induced subdiagram, contradicting minimality of $D$.
	
	By Theorem~\ref{thm-main}, $B$ is a special transversal and an exact hitting set for the maximal rectangles of $D$. Since every maximal rectangle in $E$ can be obtained as the intersection of a maximal rectangle of $D$ with $E$, it follows that $B$ is also a special transversal and exact hitting set for the maximal rectangles of $E$. But since $E$ does not branch completely with respect to $B$, we must have that for some $x \in B$, $E \backslash \{x\}$ does not branch completely. But $D \backslash \{x\}$ does branch completely, so this contradicts the minimality of $D$.
\end{proof}

We next relate $\mathcal B$ to another class of diagrams, namely that of \emph{$\Gamma$-freeable diagrams}.

\begin{defn}
	A diagram $D$ is \emph{$\Gamma$-free} if there do not exist $i_1 < i_2$ and $j_1< j_2$ such that $(i_1, j_1), (i_1, j_2), (i_2, j_1) \in D$ but $(i_2, j_2) \not\in D$. A diagram is \emph{$\Gamma$-freeable} if it is equivalent to a $\Gamma$-free diagram.
\end{defn}
There is a simple graph-theoretic characterization of $\Gamma$-freeable diagrams in terms of \emph{chordal bipartite graphs}.
\begin{defn}
	A graph $G$ is \emph{chordal bipartite} if it is bipartite and contains no induced cycle of length greater than 4.
\end{defn}
In fact, a diagram $D$ is $\Gamma$-freeable if and only if $G(D)$ is chordal bipartite. See \cite{Spinrad} for more information.

Using Proposition~\ref{prop-subdiagram}, we can easily relate complete branching to $\Gamma$-freeable diagrams.

\begin{prop} \label{prop-gammafree}
	Any completely branching diagram $D \in \mathcal B$ is $\Gamma$-freeable.
\end{prop}
\begin{proof}
	By Proposition~\ref{prop-subdiagram} and the equivalence of $\Gamma$-freeable diagrams and chordal bipartite graphs, it suffices to show that if $D$ is the diagram corresponding to an even cycle of length greater than 4, then $D$ does not branch completely. The only exact hitting sets for the maximal bicliques of such an even cycle are the two sets of alternate edges, but these are not special transversals. Hence $D$ does not branch completely by Proposition~\ref{prop-hitting}.
\end{proof}

Interestingly, exact hitting sets for maximal rectangles are automatically special transversals for $\Gamma$-freeable diagrams.
\begin{prop}
	Let $B$ be an exact hitting set for the maximal rectangles of a $\Gamma$-freeable diagram $D$. Then $B$ is a special transversal.
\end{prop}
\begin{proof}
	Suppose not. Then by Proposition~\ref{prop-alternating}, in $G(D)$ there must exist a cycle $C$ of minimum length half of whose edges lie in $B$. If $C$ has length greater than $4$, then since $G(D)$ is chordal bipartite, it must have a chord. This chord splits $C$ into two smaller cycles, one of which also has half of its edges in $B$. Thus $C$ must have length $4$, but then it corresponds to a $2 \times 2$ rectangle that contains two boxes in $B$, contradicting the fact that $B$ is an exact hitting set.
\end{proof}

Unfortunately, there exist $\Gamma$-freeable diagrams that do not branch completely. In other words, $\mathcal B$ is properly contained in the class of $\Gamma$-freeable diagrams. 

To see this, consider the following chordal bipartite graph. 
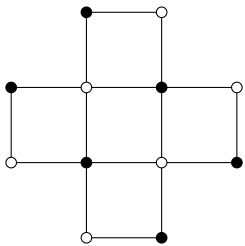
\[
\dr{(0,1)node[w]{}--(0,0)node[v]{}--(1,0)node[w]{}--(1,1)node[v]{}--(2,1)node[w]{}--(2,2)node[v]{}--(1,2)node[w]{}--(1,3)node[v]{}--(0,3)node[w]{}--(0,2)node[v]{}--(-1,2)node[w]{}--(-1,1)node[v]{}--(0,1)--(1,1)--(1,2)--(0,2)--(0,1)}
\]
The center 4-cycle is a maximal biclique, so one of its four edges must lie in any exact hitting set. But removing any of these edges creates a graph that contains an induced 6-cycle, so it cannot correspond to a diagram in $\mathcal B$ by Proposition~\ref{prop-gammafree}. Hence the original graph does not correspond to a completely branching diagram by Theorem~\ref{thm-main}.

\section{Concluding Remarks}
We have given a combinatorial criterion for determining when the Specht module $S^D$ branches completely. We have also shown that the relations in such a Specht module are generated by generalized Garnir relations.

While we have shown that the class $\mathcal B$ of completely branching diagrams contains all northwest and forest diagrams and is properly contained in the class of $\Gamma$-freeable diagrams, we are unable to give a more intrinsic description of $\mathcal B$. Ideally one would like to give a description that makes it relatively easy to check whether a diagram $D$ with $n$ boxes lies in $\mathcal B$ without having to check inductively whether $D \backslash \{x\}$ lies in $B$ for all $x \in D$. It would also be interesting to give any reasonable subclass of $\mathcal B$ that contains both northwest diagrams and forest diagrams.

The main idea behind the proof of Theorem~\ref{thm-main} is to control the relations that are needed to describe $S^D$. A lack of understanding of all relations that can occur in $S^D$ seems to be the main obstacle to applying this idea further. However, this approach may be helpful in determining when, for instance, certain exact sequences of Specht modules exist as in \cite{JamesPeel}. This approach also suggests questions such as how to classify those Specht modules $S^D$ that can be presented using only ordinary Garnir relations, or only generalized Garnir relations.

Finally, this branching rule may be helpful in determining a combinatorial formula for the irreducible decomposition of $S^D$. Such a decomposition is known for northwest diagrams but has not yet been determined for forest diagrams or most other diagrams.

\section{Acknowledgments}
The author would like to thank Victor Reiner and John Stembridge for useful conversations.

\bibliographystyle{plain}
\bibliography{branching}
\end{document}